\documentclass[11pt,a4paper]{article}

\usepackage{geometry}
\usepackage{graphicx,psfrag}

\usepackage{multirow,array,enumitem}
\usepackage{color}
\usepackage{makecell}
\usepackage{amsmath,amssymb,amsthm,url}
\usepackage{cite}

\newtheorem{theorem}{Theorem}
\newtheorem{corollary}[theorem]{Corollary}
\newtheorem{proposition}[theorem]{Proposition}
\newtheorem{conjecture}[theorem]{Conjecture}
\newtheorem{lemma}[theorem]{Lemma}
\theoremstyle{definition}
\newtheorem{definition}{Definition}
\newtheorem{obs}[theorem]{Observation}
\newtheorem{remark}[theorem]{Remark}
\newtheorem{quest}[theorem]{Question}

\definecolor{Grey}{gray}{0.5}

\DeclareMathAlphabet{\mathpzc}{OT1}{pzc}{m}{it}

\newcommand{\mc}[1]{\mathpzc{#1}}
\newcommand{\dg}{\ensuremath{{\rm dgn}}}
\newcommand{\sa}{{\rm sa}}
\newcommand{\lsa}{\ensuremath{{\rm sa}_{\ell}}}
\newcommand{\ca}{{\rm ca}}
\newcommand{\lca}{\ensuremath{{\rm ca}_{\ell}}}
\newcommand{\fca}{\ensuremath{{\rm ca}_{f}}}
\newcommand{\la}{{\rm la}}
\newcommand{\lla}{\ensuremath{{\rm la}_{\ell}}}
\newcommand{\fla}{\ensuremath{{\rm la}_{f}}}
\newcommand{\stw}{\ensuremath{{\rm stw}}}
\newcommand{\tw}{\ensuremath{{\rm tw}}}
\newcommand*{\dotcup}{\ensuremath{\mathaccent\cdot\cup}}
\newenvironment{frcseries}{\fontfamily{frc}\selectfont}{}

\newcommand{\rst}[1]{\ensuremath{{\mathbin|}%
\raise-.5ex\hbox{$#1$}}}

\title{Three ways to cover a graph}

\author{Kolja Knauer\thanks{Aix-Marseille Universit\'e, CNRS, LIF UMR 7279, 13288, Marseille, France. email:~\texttt{kolja.knauer@lif.univ-mrs.fr}}
  \and Torsten Ueckerdt\thanks{Karlsruhe Institute of Technology, Karlsruhe, Germany. email:~\texttt{torsten.ueckerdt@kit.edu}}
}

\begin{document}

\maketitle

\begin{abstract}
We consider the problem of covering an \emph{input graph} $H$ with graphs from a fixed \emph{covering class} $\mc{G}$. The classical covering number of $H$ with respect to $\mc{G}$ is the minimum number of graphs from $\mc{G}$ needed to cover the edges of $H$ without covering non-edges of $H$. We introduce a unifying notion of three covering parameters with respect to $\mc{G}$, two of which are novel concepts only considered in special cases before: the local and the folded covering number. Each parameter measures ``how far'' $H$ is from $\mc{G}$ in a different way.  Whereas the folded covering number has been investigated thoroughly for some covering classes, e.g., interval graphs and planar graphs, the local covering number has received little attention.

We provide new bounds on each covering number with respect to the following covering classes: linear forests, star forests, caterpillar forests, and interval graphs. The classical graph parameters that result this way are interval number, track number, linear arboricity, star arboricity, and caterpillar arboricity. As input graphs we consider graphs of bounded degeneracy, bounded degree, bounded tree-width or bounded simple tree-width, as well as outerplanar, planar bipartite, and planar graphs. For several pairs of an input class and a covering class we determine exactly the maximum ordinary, local, and folded covering number of an input graph with respect to that covering class.
\end{abstract}

\setcounter{page}{1}
\section{Introduction}\label{sec:int}

Graph covering is one of the most classical topics in graph theory. In 1891, in one of the first purely graph-theoretical papers, Petersen~\cite{Pet-91} showed that any $2r$-regular graph can be covered with $r$ sets of vertex disjoint cycles. A survey on covering problems by Beineke~\cite{Bei-69} appeared in 1969. Graph covering is a lively field with deep ramifications -- over the last decades as well as today~\cite{Har-70,Har-72,Aki-80,Aki-81,Gra-88,Mut-98}. This is supported through the course of this paper by many references to recent works of different authors.

In every graph covering problem one is given an input graph $H$, a covering class $\mc{G}$, and a notion of how to cover $H$ with one or several graphs from $\mc{G}$. One is then interested in $\mc{G}$-coverings of $H$ that are in some sense simple, or well structured; the most prevalent measure of simplicity being the number of graphs from $\mc{G}$ needed to cover the edges of $H$.

The main goal of this paper is to introduce the following three parameters, each of which represents how well $H$ can be covered with respect to $\mc{G}$ in a different way: 

The \emph{global covering number}, or simply \emph{covering number}, is the most classical one. It is the smallest number of graphs from $\mc{G}$ needed to cover the edges of $H$ without covering non-edges of $H$. All kinds of arboricities, e.g. star~\cite{Aki-85}, caterpillar~\cite{Gon-09}, linear~\cite{Aki-81}, pseudo~\cite{Pic-82}, and ordinary~\cite{Nas-64} arboricity of a graph are global covering numbers, where the covering class is the class of star forests, caterpillar forests, linear forests, pseudoforests, and ordinary forests, respectively. Other global covering numbers are the planar and outerplanar thickness~\cite{Bei-69,Mut-98} and the track number~\cite{Gya-95} of a graph. Here, the covering classes are planar, outerplanar, and interval graphs, respectively. 

In the \emph{local covering number} of $H$ with covering class $\mc{G}$ one also tries to cover the edges of $H$ with graphs from $\mc{G}$ but now minimizes the largest number of graphs in the covering containing a common vertex of $H$. We are aware of only two local covering numbers in the literature: {The bipartite degree} introduced by Fishburn and Hammer~\cite{Fis-96} is the local covering number where the covering class is the class of complete bipartite graphs. It was rediscovered by Dong and Liu~\cite{Don-07} as {the local biclique cover number}, and recently it has been studied in comparison with its global variant by Pinto~\cite{Pin-13}. {The local clique cover number} is another local covering number, where the covering class is the class of complete graphs. It was studied by Skums, Suzdal, and Tyshkevich~\cite{Sku-09} and by Javadi, Maleki, and Omoomi~\cite{Jav-12}. 

Finally, the \emph{folded covering number} underlies a different, but related, concept of covering. Here, one looks for a graph in $\mc{G}$ which has $H$ as homomorphic image and one minimizes the size of the largest preimage of a vertex of $H$. Equivalently, one splits every vertex of $H$ into a independent set such that the size of the largest such independent set is minimized, distributing the incident edges to the new vertices, such that the result is a graph from $\mc{G}$. The folded covering number has been investigated using interval graphs and planar graphs as covering class. In the former case the folded covering number is known as the interval number~\cite{Har-79}, in the latter case as the splitting-number~\cite{Jac-85}.

While some covering numbers, like arboricities, are of mainly theoretical interest, others, like thickness, interval number, and track number, have wide applications in VLSI design~\cite{Agg-91}, network design~\cite{Ram-92}, scheduling and resource allocation~\cite{Bar-06,But-10}, and bioinformatics~\cite{Jos-92,Jia-10a}. The three covering numbers presented here not only unify some notions in the literature, they as well seem interesting in their own right and may provide new approaches to attack classical open problems.

\begin{table}[htb]
 \centering
 \def\sm{\scriptsize}
 \renewcommand{\arraystretch}{1.1}
 \begin{tabular}{!{\vrule width 2\arrayrulewidth}>{\centering}m{1.5cm}!{\vrule width 2\arrayrulewidth}c|c!{\vrule width 2\arrayrulewidth}>{\centering}m{0.8cm}|c|>{\centering}m{1.55cm}!{\vrule width 2\arrayrulewidth}>{\centering}m{1.55cm}|c|c!{\vrule width 2\arrayrulewidth}}
  \Xcline{2-9}{2\arrayrulewidth}
  \multicolumn{1}{c!{\vrule width 2\arrayrulewidth}}{~} & \multicolumn{2}{c!{\vrule width 2\arrayrulewidth}}{star forests} & \multicolumn{3}{c!{\vrule width 2\arrayrulewidth}}{caterpillar forests} & \multicolumn{3}{c!{\vrule width 2\arrayrulewidth}}{interval graphs}\\
  \multicolumn{1}{c!{\vrule width 2\arrayrulewidth}}{~} & $g$ & $\ell=f$ & $g$ & $\ell$ & $f$ & $g$ & $\ell$ & $f$ \\
  \Xhline{2\arrayrulewidth}
  outer-planar & {\sm \textcolor{Grey}{$3$~\cite{Hak-96}}} & {\sm \textcolor{Grey}{$3$}} & {\sm \textcolor{Grey}{$3$~\cite{Kos-99}}} & {\sm \textcolor{Grey}{$3$}} & {\sm \textcolor{Grey}{$3$}} & $2$ {\sm \cite{Kos-99}} & {\sm \textcolor{Grey}{$2$}} & $2$ {\sm \cite{Sch-83}}\\
  \hline
  planar bipartite& $4$ {\sm (C\ref{cor:loc-star-planar})} & $3$ {\sm (C\ref{cor:loc-star-planar})} & $4$ {\sm \cite{Gon-07}} & {\sm \textcolor{Grey}{$3$}} & {\sm \textcolor{Grey}{$3$}} & {\sm \textcolor{Grey}{$4$}} & {\sm \textcolor{Grey}{$3$}} & $3$ {\sm \cite{Sch-83}}\\
  \hline
  planar & $5$ {\sm \cite{Alg-89,Hak-96}} & $4$ {\sm (C\ref{cor:loc-star-planar})} & $4$ {\sm \cite{Gon-07}} & {\sm \textcolor{Grey}{$4$}} & $4$ {\sm \cite{Sch-83}} & $4$ {\sm \cite{Gon-09}} & $?$ & $3$ {\sm \cite{Sch-83}}\\
  \Xhline{2\arrayrulewidth}
  $\stw\leq k$ & {\sm \textcolor{Grey}{$k$+$1$}} & {\sm \textcolor{Grey}{$k$+$1$}} & {\sm \textcolor{Grey}{$k$+$1$}} & {\sm \textcolor{Grey}{$k$+$1$}} & $k$+$1$ {\sm (T\ref{thm:fca-stw-lower})} & $k$+$1$ {\sm (T\ref{thm:t-stw-lower})} & $k$ {\sm (T\ref{thm:tl-stw-upper})} & $k$ {\sm (T\ref{thm:i-tw-lower})} \\
  \hline
  $\tw \leq k$ & $k$+$1$ {\sm \cite{Din-98,Duj-07}} & {\sm \textcolor{Grey}{$k$+$1$}} & {\sm \textcolor{Grey}{$k$+$1$}} & {\sm \textcolor{Grey}{$k$+$1$}} & {\sm \textcolor{Grey}{$k$+$1$}} & {\sm \textcolor{Grey}{$k$+$1$}} & {\sm \textcolor{Grey}{$k$+$1$}} & $k$+$1$ {\sm (T\ref{thm:i-tw-lower})}\\
  \hline
  $\dg \leq k$ & $2\,k$ {\sm \cite{Alo-92}} & $k$+$1$ {\sm (C\ref{cor:lsa-deg-upper})} & {\sm \textcolor{Grey}{$2\,k$}} & {\sm \textcolor{Grey}{$k$+$1$}} & {\sm \textcolor{Grey}{$k$+$1$}} & $2\,k$ {\sm (T\ref{thm:t-deg-lower})} & {\sm \textcolor{Grey}{$k$+$1$}} & {\sm \textcolor{Grey}{$k$+$1$}} \\
  \Xhline{2\arrayrulewidth}
 \end{tabular}\label{table:1}
 \vspace{2pt}
 \caption{Overview of results.}
\end{table}

In this paper we moreover present new lower and upper bounds for several covering numbers. In the new results, the covering classes are: interval graphs, star forests, linear forests, and caterpillar forests. The input classes are: graphs of bounded degeneracy, bounded tree-width or bounded simple tree-width, as well as outerplanar, planar bipartite, planar, and regular graphs. Not all pairs of these input classes with these covering classes are given new bounds. We provide an overview over some of our new results in Table~\ref{table:1}. Each row of the table corresponds to an input class $\mc{H}$, each column to a covering class $\mc{G}$. Every cell contains the maximum covering number among all graphs $H \in \mc{H}$ with respect to the covering class $\mc{G}$, where the columns labeled $g, \ell, f$ stand for the global, local, and folded covering number, respectively. Grey entries follow by Proposition~\ref{prop:basic} from other stronger results in the table. Letters T and C stand for Theorem and Corollary in the present paper, respectively. Indeed all the entries except the '?' in Table~\ref{table:1} are exact, with matching upper and lower bounds. Note that besides results we prove as new theorems as indicated, many values in the table (written in gray) follow from the point of view offered by our general approach (Proposition~\ref{prop:basic}). 

This paper is structured as follows:
In order to give a motivating example before the general definition, we start by discussing in Section~\ref{sec:linear arboricity} the linear arboricity and its local and folded variants. In Section~\ref{sec:covers-and-numbers} the three covering numbers are formally introduced and some general properties are established. In Section~\ref{sec:guest} we introduce the covering classes star forests, caterpillar forests, and interval graphs, and in Section~\ref{sec:results} we present our results claimed in Table~\ref{table:1}. In Section~\ref{sec:complexity} we briefly discuss the computational complexity of some covering numbers, giving a polynomial-time algorithm for the local star arboricity. Moreover, we discuss by how much global, local and folded covering numbers can differ.

For the entire paper we assume all graphs to be simple without loops nor multiple edges. Notions used but not introduced can be found in any standard graph theory book; such as~\cite{Wes-96}.

\section{Global, Local, and Folded Linear Arboricity}\label{sec:linear arboricity}

We give the general definitions of covers and covering numbers in Section~\ref{sec:covers-and-numbers} below. In this section we motivate and illustrate these concepts on the basis of one fixed covering class: the class $\mc{L}$ of \emph{linear forests}, which are the disjoint unions of paths. We want to cover an input graph $H$ by several linear forests $L_1,\ldots,L_k \in \mc{L}$. That is, every edge $e \in E(H)$ is contained in at least\footnote{Since linear forests are closed under taking subgraphs, we can indeed assume that $e \in L_i$ for exactly one $i \in [k]$.} one $L_i$ and no non-edge of $H$ is contained in any $L_i$. When $H$ is covered by $L_1,\ldots,L_k$ we write $H = \bigcup_{i \in [k]}L_i$.

The \emph{linear arboricity of $H$}, denoted by $\la(H)$, is the minimum $k$ such that $H = \bigcup_{i \in [k]} L_i$ and $L_i \in \mc{L}$ for $i \in [k]$. One easily sees that every graph $H$ of maximum degree $\Delta(H)$ has $\la(H) \geq \left\lceil \frac{\Delta(H)}{2} \right\rceil$, and every $\Delta(H)$-regular graph $H$ has $\la(H) \geq \left\lceil \frac{\Delta(H)+1}{2} \right\rceil$. In 1980, Akiyama \textit{et al.}~\cite{Aki-81} stated the Linear Arboricity Conjecture (LAC). It says that the linear arboricity of any simple graph $H$ of maximum degree $\Delta(H)$ is either $\left\lceil\frac{\Delta(H)}{2}\right\rceil$ or $\left\lceil\frac{\Delta(H)+1}{2}\right\rceil$. LAC was confirmed for planar graphs by Wu and Wu~\cite{Wu-99,Wu-08} and asymptotically for general graphs by Alon and Spencer~\cite{Alo-08}. The general conjecture remains open. The best-known general upper bound for $\la(H)$ is $\left\lceil \frac{3\Delta(H)+2}{5}\right\rceil$, due to Guldan~\cite{Gul-86}.

We define the \emph{local linear arboricity of $H$}, denoted by $\lla(H)$, as the minimum $j$ such that $H = \bigcup_{i \in [k]} L_i$ for some $k$ and every vertex $v$ in $H$ is contained in at most $j$ different $L_i$. Again, if $H$ has maximum degree $\Delta(H)$, then $\lla(H) \geq \left\lceil \frac{\Delta(H)}{2} \right\rceil$, and if $H$ is $\Delta(H)$-regular, then $\lla(H) \geq \left\lceil \frac{\Delta(H)+1}{2} \right\rceil$. Note that $\lla(H)$ is at most $\la(H)$, and hence the following statement must necessarily hold for LAC to be true.

\begin{conjecture}{Local Linear Arboricity Conjecture (LLAC):}\label{conj:LLAC}
 The local linear arboricity of any simple graph with $H$ maximum degree $\Delta(H)$ is either $\left\lceil \frac{\Delta(H)}{2} \right\rceil$ or $\left\lceil \frac{\Delta(H)+1}{2} \right\rceil$.
\end{conjecture}

\begin{obs}
 To prove LAC or LLAC it suffices to consider regular graphs of odd degree: Regularity is obtained by considering a $\Delta(H)$-regular supergraph of $H$. If $\Delta(H)$ is even, say $\Delta(H) = 2k$, one can find a spanning linear forest $L_{k+1}$ in $H$~\cite{Gul-86}, remove it from the graph, and extend $L_{k+1}$ by a cover $L_1,\ldots,L_k$ in the remaining graph of maximum degree $\Delta(H)-1 = 2k-1$. 
\end{obs}

If $H$ is regular with odd degree, then LLAC states that $H = \bigcup_{i \in [k]} L_i$ with every vertex being an endpoint of exactly one path. LAC additionally requires that the paths can be colored with $\left\lceil \frac{\Delta(H)}{2} \right\rceil$ colors such that no two paths that share a vertex receive the same color. We will see in later sections that sometimes the coloring is the crucial and difficult task.

Next we propose a second way to cover the input graph $H$ with linear forests. A \emph{walk} in $H$ is a sequence of consecutively incident edges of $H$ of the form $\{v_1v_2, v_2v_3, \ldots, v_{k-1}v_k\}$ for $v_1, \ldots v_k$ being vertices of $H$. As before, a set $W_1,\ldots,W_k$ of walks covers $H$, denoted by $H = \bigcup_{i \in [k]}W_i$, if the edge-set $E$ of $H$ is the union of the edge-sets of the walks. We are now interested in how often a vertex $v$ in $H$ appears in the walks $W_1,\ldots,W_k$ in total. The \emph{folded linear arboricity of $H$}, denoted by $\fla(H)$, is the minimum $j$ such that $H = \bigcup_{i \in [k]}W_i$ and every vertex $v$ in $H$ appears at most $j$ times in the walks $W_1,\ldots,W_k$. Again if $H$ has maximum degree $\Delta(H)$ then $\fla(H) \geq \left\lceil \frac{\Delta(H)}{2} \right\rceil$, and if $H$ is $\Delta(H)$-regular then $\fla(H) \geq \left\lceil \frac{\Delta(H)+1}{2} \right\rceil$. Clearly, $\fla(H) \leq \lla(H)$. The next theorem follows directly from a short proof of West~\cite{Wes-89} of a result previously published by Griggs and West~\cite{Gri-80} (where it is stated in terms of the interval number $i(H)$). It is a weakening of LLAC 
above.

\begin{theorem}\label{thm:FLAC}
 If $H$ has maximum degree $\Delta(H)$ then $\fla(H) \in \{\left\lceil \frac{\Delta(H)}{2}\right\rceil, \left\lceil \frac{\Delta(H)+1}{2}\right\rceil\}$.
\end{theorem}
\begin{proof}
 Add a vertex $x$ to $H$ and connect it to every vertex in $H$ of odd degree. Each component of the resulting graph is Eulerian. Consider any Eulerian tour in $H \cup x$ (or $H$) and split it into shorter walks by removing $x$ from it.
\end{proof}

\section{Covers and Covering Numbers}\label{sec:covers-and-numbers}

In this section we formalize the concepts from Section~\ref{sec:linear arboricity} with respect to general covering and input classes and obtain some general inequalities. The notation we introduce is convenient for making our generalized approach as transparent as possible. When treating concrete covering classes, for which covering numbers already have an established notation in the literature later on in the paper, we will use the latter in order to make results more accessible to readers already familiar with the parameters.

A \emph{homomorphism} from a graph $G$ to a graph $H$ is a map $\varphi:V(G)\to V(H)$ such that $vw\in E(G)$ implies $\varphi(v)\varphi(w)\in E(H)$. We call a homomorphism \emph{edge-surjective} if for all $v'w'\in E(H)$ there exists $vw\in E(G)$ such that $\varphi(v)=v'$ and $\varphi(w)=w'$. For an input graph $H$ and a covering class $\mc{G}$, we define a \emph{$\mc{G}$-cover of $H$} as an edge-surjective homomorphism $\varphi:G_1 \dotcup G_2 \dotcup \cdots \dotcup G_k\to H$, where $G_i \in \mc{G}$ for $i \in [k]$ and $\dotcup$ denotes the vertex disjoint union. The \emph{size} of a cover is the number of covering graphs in the disjoint union. A cover $\varphi$ is called \emph{injective} if $\varphi\rst{G_i}$, that is, $\varphi$ restricted to $G_i$, is injective for every $i \in [k]$.

\begin{definition}\label{defn:numbers}
 For a covering class $\mc{G}$ and an input graph $H$ define the \emph{(global) covering number} $c_{g}^{\mc{G}}(H)$, the \emph{local covering number} $c_{\ell}^{\mc{G}}(H)$, and the \emph{folded covering number} $c_{f}^{\mc{G}}(H)$ as follows:
 \begin{itemize}[label=]
  \item $c_{g}^{\mc{G}}(H) = \min\left\{ \text{size of }\varphi : \varphi \text{ is an injective }\mc{G}\text{-cover of }H\right\}$
  \item $c_{\ell}^{\mc{G}}(H) = \min\left\{ \max_{v \in V(H)} |\varphi^{-1}(v)| : \varphi \text{ is an injective }\mc{G}\text{-cover of }H\right\}$
  \item $c_{f}^{\mc{G}}(H) = \min\left\{ \max_{v \in V(H)} |\varphi^{-1}(v)| : \varphi \text{ is a }\mc{G}\text{-cover of }H \text{ having size } 1\right\}$
 \end{itemize}
\end{definition}

Let us rephrase $c_{g}^{\mc{G}}(H)$, $c_{\ell}^{\mc{G}}(H)$, and $c_{f}^{\mc{G}}(H)$. The covering number is the minimum number of graphs in $\mc{G}$ needed to cover $H$ exactly, where \emph{covering exactly} means identifying subgraphs in $H$ that are covering graphs, such that every edge of $H$ is contained in some covering graph. In the local covering number the number of covering graphs is not restricted; instead the number of covering graphs at every vertex should be small. We will see later that these two numbers can differ significantly. The folded covering number is the minimum $k$ such that every vertex $v$ of $H$ can be split into at most $k$ vertices, distributing the incident edges at $v$ arbitrarily (even repeatedly) among them, such that the resulting graph belongs to $\mc{G}$. The splitting corresponds to representing the vertex by the set of its preimages under the edge-surjective homomorphism $\varphi$.

One is often interested in the maximum or minimum value of a graph parameter on a class of input graphs. For $i \in \{g,\ell,f\}$, a covering class $\mc{G}$, and an input graph class $\mc{H}$, we define $c_{i}^{\mc{G}}(\mc{H}) = \sup\left\{ c_{i}^{\mc{G}}(H) \colon\ H \in \mc{H} \right\}$. We close this section with a list of inequalities, most of which are elementary applications of Definition~\ref{defn:numbers} and homomorphisms.

\begin{proposition}\label{prop:basic}
 For covering classes $\mc{G},\mc{G}'$, input classes $\mc{H},\mc{H}'$ and any input graph $H$ we have the following:
 \begin{enumerate}[label = (\roman*)]
  \item $c_g^{\mc{G}}(H)\geq c_{\ell}^{\mc{G}}(H)$, and if $\mc{G}$ is closed under disjoint union, then $c_{\ell}^{\mc{G}}(H) \geq c_f^{\mc{G}}(H)$\label{enum:inequalities}.
  \item If $\mc{G}$ is closed under merging non-adjacent vertices within connected components (and afterwards deleting multiple edges) and restriction to maximal connected components, then $c_{\ell}^{\mc{G}}(H)\leq c_f^{\mc{G}}(H)$.\label{enum:local-eq-folded} 
  \item If $\mc{H}\subseteq \mc{H}'$, then $c_i^{\mc{G}}(\mc{H})\leq c_i^{\mc{G}}(\mc{H}')$ for $i\in\{g,\ell,f\}$.\label{enum:subclass-host}
  \item If $H_{\mc{G}}$ and $H_{\mc{G}'}$ denote the set of subgraphs of $H$ that are homomorphic images of graphs in $\mc{G}$ and $\mc{G}'$, respectively, then $H_{\mc{G}}\subseteq H_{\mc{G}'}$ implies $c_i^{\mc{G}}(H) \geq c_i^{\mc{G}'}(H)$ for $i\in\{g,\ell,f\}$. This holds in particular when $\mc{G}\subseteq \mc{G}'$.\label{enum:subclass-guest}
  \item If $\bar{H}$ denotes the set of all subgraphs of $H$ and we have $\mc{G}\cap\bar{H} \subseteq \mc{G}'\cap\bar{H}$, then $c_i^{\mc{G}}(H) \geq c_i^{\mc{G}'}(H)$ for $i\in\{g,\ell\}$.\label{enum:intersection-gl}
 \end{enumerate}
\end{proposition}
\begin{proof}
 The first inequality in~\ref{enum:inequalities} follows from the definition, the second one comes by viewing an injective cover $G_1 \dotcup G_2 \dotcup \cdots \dotcup G_k$ as a $\mc{G}$-cover of size $1$.
 
 To see~\ref{enum:local-eq-folded}, let $\varphi:G\to H$ be a $\mc{G}$-cover of $H$ of size $1$ witnessing $c_f^{\mc{G}}(H)$. Now for every $v\in H$ and a component $G'$ of $G$ merge all $\varphi^{-1}(v)\cap V(G')$ into one vertex (and delete multiple edges). Since $H$ has no loops, the merging process creates no loops. Doing this for all components of $G$ yields a new covering graph $\widetilde{G}\in\mc{G}$ with homomorphism $\widetilde{\varphi}$ being injective on each component. Clearly, $|\widetilde{\varphi}^{-1}(v)|\leq|\varphi^{-1}(v)|$.
 
 Claims~\ref{enum:subclass-host} and~\ref{enum:subclass-guest} follow immediately from the definition. To see~\ref{enum:intersection-gl} note that it follows similarly as~\ref{enum:subclass-guest}, because $\mc{G}\cap\bar{H}$ and $\mc{G}'\cap\bar{H}$ are the subgraphs of $H$ that arise as images of \emph{injective} covers.
\end{proof}

\begin{remark}
Within the scope of this paper {we only consider covering classes that are closed under disjoint union} even without explicitly saying so. For example, when considering stars or complete graphs as covering graphs, we actually mean star forests and disjoint unions of complete graphs, respectively. If the covering class $\mc{G}$ is closed under disjoint union, then the restriction to covers of size $1$ in the definition of $c_{f}^{\mc{G}}$ is unnecessary. 

 It is still interesting to consider covering classes that are {not} closed under disjoint union. Haj\'os' Conjecture~\cite{Lov-68} states that the edges of any $n$-vertex Eulerian graph $H$ may be partitioned into $\lfloor \frac{n}{2} \rfloor$ cycles. Haj\'os' Conjecture being widely open, one may consider {coverings} with cycles. When $\mc{C'}$ denotes the class of all simple cycles and $H$ is an $n$-vertex Eulerian graph, Fan~\cite{Fan-03} proved $c_{g}^{\mc{C'}}(H) \leq \lfloor \frac{n-1}{2} \rfloor$. 
\end{remark}

\subsubsection*{Example}
 
In order to illustrate the notions introduced above, consider the covering class $\mc{C}$ of disjoint unions of cycles. As input graph $H$ we take the Petersen graph. See Figure~\ref{fig:petersen} where we have from left to right: A global cover with three unions of cycles, a local cover of size five with at most three cycles at each vertex, and a folded cover with two preimages per vertex. Note that the local cover does not yield an optimal global cover.

\begin{figure}[htb]
  \centering
  \includegraphics[width=\textwidth]{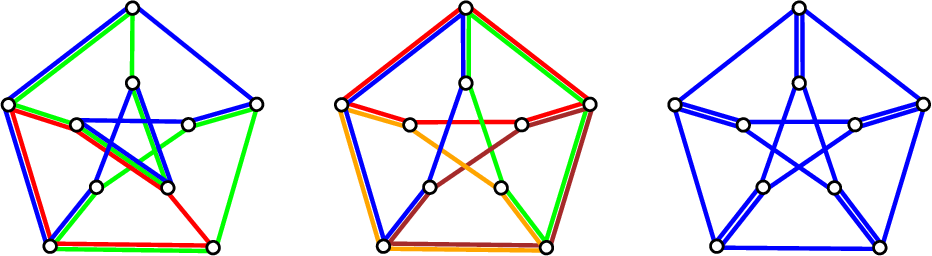}
  \caption{Coverings of the Petersen graph by disjoint unions of cycles.}
  \label{fig:petersen}
 \end{figure}

\begin{proposition}\label{prop:petersen}
For the Petersen graph, we have $3=c_g^{\mc{C}}(H)=c_{\ell}^{\mc{C}}(H)>c_{f}^{\mc{C}}(H)=2$.                                                                                                                                                                     
\end{proposition}
\begin{proof}
 All witnesses for the upper bounds are shown in Figure~\ref{fig:petersen}.
 Clearly, $c_{f}^{\mc{C}}(H)\geq 2$ since otherwise $H$ would have to be a disjoint union of cycles. Now suppose, $c_{\ell}^{\mc{C}}(H)=2$. Since $H$ is cubic, at each vertex there is exactly one edge contained in two cycles of the covering. Thus, these edges form a perfect matching $M$ of $H$. Moreover, all cycles involved in the cover are alternating cycles with respect to  $M$. In particular they are all even and of length $6$ or $8$ (as this graph is not Hamiltonian there is no $10$-cycle). Since $M$ is covered twice and the remaining edges of $H$ once, the sum of sizes of cycles in the cover is $20$, which can be obtained only as $6+6+8$. In particular, a $6$-cycle $C$ must be involved. Now $M$ restricted to $H\backslash V(C)$ is still a perfect matching, but $H\backslash V(C)$ is a claw.
\end{proof}

\section{Covering Classes}\label{sec:guest}
In this section we introduce the covering classes and covering numbers corresponding to the columns of Table~\ref{table:1}. We also include some known results and general observations.

\subsection{Forests and Pseudoforests}\label{subsec:pseudoforest}

 Nash-Williams~\cite{Nas-64} showed that the minimum number of forests needed to cover the edges of $H$ is $\max_{S\subseteq V(H)}\left\lceil\frac{|E[S]|}{|S|-1}\right\rceil$, where E[S] denotes the set of edges in the subgraph induced by S.. This value, denoted by $a(H)$, is now usually called the \emph{arboricity} of $H$, see Beineke~\cite{Bei-69} for an early appearance of this name. Clearly, $a(H)=c_g^{\mc{G}}(H)$, where $\mc{G}$ is the class of forests.

A \emph{pseudoforest} is a graph with at most one cycle per component and the \emph{pseudoarboricity} $p(H)$ is the minimum number of pseu\-do\-fo\-rests needed to cover the edges of $H$. Thus, $p(H)=c_g^{\mc{G}}(H)$, where $\mc{G}$ is the class of pseudoforests. 
Results of Picard and Queyranne~\cite{Pic-82} and Frank and Gy\'arf\'as~\cite{Fra-78} yield the following lemma.

\begin{lemma}[\cite{Fra-78,Pic-82}]\label{lem:pseudo}
 The pseudoarboricity $p(H)$ of a graph $H$ equals the minimum over all orientations of $H$ of the maximum out-degree of $H$. Furthermore, $p(H)=\max_{S\subseteq V(H)}\left\lceil\frac{|E[S]|}{|S|}\right\rceil$.
\end{lemma}

Using $a(H)=\max_{S\subseteq V(H)}\left\lceil\frac{|E[S]|}{|S|-1}\right\rceil$, one immediate consequence of Lemma~\ref{lem:pseudo} is $p(H) \leq a(H) \leq p(H)+1$. 

\begin{theorem}\label{thm:lf-(pseudo)arboricity}
For every graph, the values of global, local, and folded (pseudo)arboricity coincide.
\end{theorem}
\begin{proof}
 Take a folded covering $\varphi$ of $H$ with a (pseudo)forest, such that for every $v\in H$ we have $|\varphi^{-1}(v)|\leq c$. Since (pseudo)forests are closed under taking induced subgraphs, this in particular yields a covering for every induced subgraph $H[S]$ such that every vertex is covered at most $c$ times. Now, focusing on pseudoforests, we know that the subgraph of the covering graph induced by $\varphi^{-1}(S)$ has at most $c|S|$ edges, and therefore $c|S|\geq |E[S]|$, i.e., $c\geq \left\lceil\frac{|E[S]|}{|S|}\right\rceil$. Now by Lemma~\ref{lem:pseudo}, we have  $p(H)=\max_{S\subseteq V(H)}\left\lceil\frac{|E[S]|}{|S|}\right\rceil$ yielding the result for folded coverings. Now, Proposition~\ref{prop:basic}\ref{enum:inequalities} gives the result for the local covering number.
 
 Along the same lines one obtains $c\geq \left\lceil\frac{|E[S]|+1}{|S|}\right\rceil$ when $c$ is the number of times a vertex is covered in a forest-cover of $H$. It is then easy to compute $\left\lceil\frac{|E[S]|+1}{|S|}\right\rceil=\left\lceil\frac{|E[S]|}{|S|-1}\right\rceil$, since $|E[S]|\leq \binom{|S|}{2}$. The result follows as in the case of pseudoarboricity.
\end{proof}

\subsection{Star Forests}\label{subsec:star}
The star arboricity $\sa(H)$ of a graph $H$, introduced by Akiyama and Kano~\cite{Aki-85}, is the minimum number of star forests (forests without paths of length $3$) into which the edge-set of $H$ can be partitioned. In particular, if $\mc{S}$ denotes the class of star forests, then $\sa(H) = c_{g}^{\mc{S}}(H)$. The star arboricity has been a frequent subject of research. It is known that outerplanar and planar graphs have star arboricity at most $3$ and $5$, respectively; see Hakimi \textit{et al.}~\cite{Hak-96}. That this is best possible was shown by Algor and Alon~\cite{Alg-89}. 
Alon \textit{et al.}~\cite{Alo-92} showed that $sa(H)\leq 2a(H)$ is a tight upper bound. 

Since merging non-adjacent vertices in a star and omitting  double edges yields again a star, local and folded star arboricity coincide, by Proposition~\ref{prop:basic}~\ref{enum:local-eq-folded}. Here, we show that in contrast to the global star arboricity, the local star arboricity, denoted by $\lsa(H)$, fits nicely into the inequalities relating arboricity and pseudoarboricity from Section~\ref{subsec:pseudoforest}.

\begin{theorem}\label{thm:starorient}
 For any graph $H$, we have $p(H) \leq a(H) \leq \lsa(H) \leq p(H)+1$, where any inequality can be strict. Moreover, $\lsa(H) = p(H)$ if and only if $H$ has an orientation with maximum out-degree $p(H)$ in which this outdegree occurs only at vertices of degree $p(H)$.
\end{theorem}
\begin{proof}
 Every cover of $H$ with respect to stars can be transformed into an orientation of $H$ by orienting every edge towards the center of the corresponding star. If every vertex is contained in at most $\lsa(H)$ stars, then the orientation has maximum out-degree at most $\lsa(H)$. Lemma~\ref{lem:pseudo} then gives $p(H) \leq \lsa(H)$.

 In the same way, every orientation can be transferred into a cover with respect to stars by taking at every vertex the star of its incoming edges. If the orientation has maximum out-degree $p(H)$, then each vertex is contained in no more than $p(H)+1$ stars, i.e., $\lsa(H) \leq p(H)+1$. Moreover, the maximum out-degree is $\lsa(H)$ if and only if for every vertex $v$ lying in $\lsa(H)$ stars with centers different from $v$ there is no star with center $v$. Equivalently, $\lsa(H) = p(H)$ if and only if the maximum out-degree $p(H)$ is attained only at vertices of degree $p(H)$.

 If $sa_l(H)=p(H)+1$, then $a(H)\leq sa_l(H)$ follows from $a(H)\leq p(H)-1$.  When $sa_l(H)=p(H)$, there is an orientation with maximum out-degree $p(H)$ attained only at vertices with degree $p(H)$. Removing these vertices, we obtain a graph $H'$ with $p(H')\leq p(H)-1$, in particular $a(H')\leq p(H)$. We reinsert the vertices of degree $p(H)$ putting each incident edge into a different one of the $p(H)$ forests that partition $H'$. We obtain a cover of $H$ with $p(H)$ forests, so $a(H) \leq p(H) = \lsa(H)$. 

 Finally, we show that each inequality can be strict: First $k = p(H) < a(H)$ holds for every $2k$-regular graph $H$, due to the number of edges of the covering graphs. Second, we claim that $k = p(H) = \lsa(H)$ holds for the complete bipartite graph $K_{k,n}$ with $n$ large enough. Indeed, $p(K_{k,n})=\max_{S\subseteq V(K_{k,n})}\left\lceil\frac{|E[S]|}{|S|}\right\rceil = \left\lceil \frac{kn}{k+n}\right\rceil = k$, and taking all maximal stars with centers in the smaller class of the bipartition yields $\lsa(K_{k,n}) \leq k$.

 It remains to present a graph $H$ with $k = a(H) < \lsa(H)$. We take $H$ to be the $k$-dimensional grid of size $m$. That is, $V(H)=[m]^k$, and there is an edge joining vertices $v$ and $w$ if and only if they differ in exactly one coordinate and differ there by $1$. It is straightforward to compute that $H$ has $(m-1)m^{k-1}k$ edges. Observing that $H$ itself is a densest induced subgraph, the formulas for arboricity and pseudoarboricity give $a(H)=p(H)=k$ for large enough $m$. Also, $a(H)=\lsa(H)$ implies $p(H)=\lsa(H)$. Hence, as proved above, $H$ has an orientation with maximum out-degree $k$, which furthermore is only attained at vertices of degree $k$. However, $H$ has only $2^k$ vertices of degree $k$. If all other vertices have outdegree at most $k-1$, then $H$ has at most $2^kk+(m^k-2^k)(k-1)$ edges. Choosing $m>2^k+k$ yields a contradiction to the number of edges of $H$ calculated above.
\end{proof}

We will derive from Theorem~\ref{thm:starorient} tight upper bounds for the local star arboricity in Section~\ref{sec:results}, as well as a polynomial-time algorithm to compute the local star arboricity in Section~\ref{sec:complexity}.

\subsection{Other Covering Classes}
\subsubsection{Caterpillar Forests}\label{subsec:caterpillar}
A graph parameter related to the star arboricity is the \emph{caterpillar arboricity} $\ca(H)$ of $H$. A \emph{caterpillar} is a tree in which all non-leaf vertices form a path, called the \emph{spine}. The caterpillar arboricity is the minimum number of caterpillar forests into which the edge-set of $H$ can be partitioned. It has mainly been considered for outerplanar graphs (Kostochka and West~\cite{Kos-99}), and for planar graphs (Gon{\c{c}}alves and Ochem~\cite{Gon-07,Gon-09}).

\subsubsection{Interval Graphs}\label{subsec:interval}

The class $\mathcal{I}$ of \emph{interval graphs} has already been considered in many ways and remains present in today's literature. Interval graphs have been generalized to intersection graphs of systems of intervals by several groups of people: Gy{\'a}rf{\'a}s and West~\cite{Gya-95} proposed the $\mathcal{I}$-covering and introduced the corresponding global covering number called the \emph{track number}, denoted by $t(H)$, i.e., $t(H) = c_{g}^{\mathcal{I}}(H)$. It has been shown that outerplanar and planar graphs have track number at most $2$~\cite{Kos-99} and $4$~\cite{Gon-09}, respectively. Already in 1979, Harary and Trotter~\cite{Har-79} introduced the folded $\mathcal{I}$-covering number, called the \emph{interval number}, denoted by $i(H)$, i.e., $i(H) = c_{f}^{\mathcal{I}}(H)$. It is known that trees have interval number at most $2$~\cite{Har-79}. Also, outerplanar and planar graphs have interval number at most~$2$ and $3$, respectively, see Scheinermann and West~\cite{Sch-83}. All these bounds are tight.

The \emph{local track number} $t_{\ell}(H) := c_{\ell}^{\mathcal{I}}(H)$ is a natural variation of $i(H)$ and $t(H)$, which to our knowledge has not been considered so far.


\section{Results}\label{sec:results}
In this section we present all the new results displayed in Table~\ref{table:1}. We proceed input class by input class.

\subsection{Bounded Degeneracy}\label{subsec:degeneracy}
The \emph{degeneracy} $\dg(H)$ of a graph $H$ is the minimum of the maximum out-degree over all {acyclic} orientations of $H$. It 
is a classical measure for the sparsity of $H$. By Lemma~\ref{lem:pseudo} and the definition we have $p(H)\leq a(H)\leq \dg(H)$. Thus, the next corollary follows directly from Theorem~\ref{thm:starorient}.

\begin{corollary}\label{cor:lsa-deg-upper}
 For every $H$ we have $\lsa(H)\leq\dg(H)+1$.
\end{corollary}

Let $\mathcal{I}$ be the class of interval graphs and $\mc{Ca}$ be the class of caterpillar forests, i.e., the class of {bipartite} interval graphs. Since homomorphisms the image of a homomorphism has chromatic number at least as large as its preimage, the chromatic number of an interval graph $G$ that has a bipartite homomorphic image is at most two. Thus, $G$ is a caterpillar forest. Therefore, when $G$ is bipartite, the set of all homomorphic images of caterpillar forests in $G$ coincides with the set of all homomorphic images of interval graphs in $G$. Thus, by Proposition~\ref{prop:basic}~\ref{enum:subclass-guest} we have $c_{i}^{\mathcal{I}}(H) = c_{i}^{\mc{Ca}}(H)$ for $i\in\{g,\ell,f\}$ for every {bipartite} graph $H$. In particular, if $H$ is bipartite then $t(H) = \ca(H)$ and $i(H) = \fca(H)$. In 
the remainder of this section 
we present graphs with high (folded) caterpillar arboricity. Since all these graphs are bipartite, we obtain lower bounds on the track number and interval number of those graphs. Indeed in all constructions we define a supergraph $H$ of the complete bipartite graph $K_{m,n}$. The track number and interval number of $K_{m,n}$ have already been determined: $t(K_{m,n}) = \ca(K_{m,n}) = \left\lceil \frac{mn}{m+n-1}\right\rceil$~\cite{Gya-95} and $i(K_{m,n}) = \fca(K_{m,n}) = \left\lceil\frac{mn+1}{m+n}\right\rceil$~\cite{Har-79}. 

In order to formulate the following lemma, we need to introduce one more notion. For a cover $\varphi$ of $H$ by $G_1\dotcup\ldots\dotcup G_k$ with $G_i\in\mc{G}$ and a subgraph $H'$ of $H$, we define the \emph{restriction of $\varphi$ to $H'$} as a cover $\psi$ of $H'$ by $G'_1\dotcup\ldots\dotcup G'_k$, where $G'_i$ comes from $G_i$ by deleting $\{e\in E(G_i)\colon\ \varphi(e)\notin H'\}$ and then by removing isolated vertices. The resulting mapping $\psi$ is the restriction of $\varphi$ to $G'_1\dotcup\ldots\dotcup G'_k$. If $\mc{G}$ is closed under taking subgraphs, then $\psi$ is also a $\mc{G}$-cover. Note that while restriction of a function normally means its specification on a subset of the domain,  here we are restricting the image, which turn induces a restriction of the domain. 

To increase readability we refer to the classes of size $m$ and $n$ in the bipartition of $K_{m,n}$  by $A$ and $B$, respectively.

\begin{lemma}\label{lem:Kmn-caterpillar}
 Let $H$ be a graph with an induced $K_{m,n}$ and $\varphi$ be a $\mc{Ca}$-cover of $H$ with $s = \max\{ |\varphi^{-1}(a)| \colon\ a \in A\}$. If $\psi$ is the restriction of $\varphi$ to the subgraph $H'$ of $H$ after removing all edges in $K_{m,n}$, then there are at least $n - 2sm$ vertices $b \in B$ such that $|\psi^{-1}(b)| \leq |\varphi^{-1}(b)| - m$.
\end{lemma}
\begin{proof}

 Every $a\in A$ is the image of at most $s$ vertices among $C_1\dotcup\ldots\dotcup C_k$. Denote by $s'$ the number of vertices in $\varphi^{-1}(a)$ that are incident to two spine-edges and by $s''$ the number of vertices in $\varphi^{-1}(a)$ that are leaves. Clearly, $s'+s''\leq s$. Moreover, at most $2s'+s''$ edges incident to $a$ are covered by spine-edges or edges whose degree $1$ vertex is mapped to $a$. Therefore, at least $n-2s$ edges at $a$ have to be covered under $\varphi$ by a non-spine edge with a vertex $b$ being the image of a leaf. Thus, for at least $n-2sm$ vertices $b \in B$ this is the case with respect to {every} $a \in A$.
 

 Now if $e=ab$ is covered by some edge in $C_i$ with $b$ being a leaf, then in the restriction of $\varphi$ to $H \setminus e$ the number of preimages of $b$ is one less than in $\varphi$. This concludes the proof.
\end{proof}

\begin{theorem}\label{thm:t-deg-lower}
 For $k\geq 1$ there is a bipartite graph $H$ such that
 \[2\dg(H)\leq 2k \leq \ca(H) = t(H).\]
\end{theorem}
\begin{proof}
To construct $H$, begin with a copy of $K_{k,n}$ having $|A|=k$ and $|B|=n$ with $n>(k-1)\binom{2k-1}{k-1} + 2k(2k-1)$. For each $k$-subset $S$ of $B$, add $(k-1)^2+1$ new vertices $B_S$ with neighborhood $S$.  The resulting graph $H$ is bipartite with every vertex in $A$ and $B_S$ for any $S$ having degree $k$, so $\dg(H)=k$.


 Now consider an injective $\mc{Ca}$-cover $\varphi$ of $H$ and its restriction $\psi$ to the subgraph of $H$ after removing all edges in $K_{k,n}$. Assume for the sake of contradiction that the size $s$ of $\varphi$ is at most $2k-1$, i.e., $\max\{|\varphi^{-1}(v)| : v \in V(H)\} = s \leq 2k-1$. Then by Lemma~\ref{lem:Kmn-caterpillar}, there is a set $W \subset B$ of at least $n-2(2k-1)k > (k-1)\binom{2k-1}{k-1}$ vertices such that $|\psi^{-1}(b)| \leq |\varphi^{-1}(b)| - k \leq s - k \leq k-1$ for every $b \in W$. In other words, every $b \in W$ has a preimage under $\psi$ in at most $k-1$ of the $2k-1$ caterpillar forests. Since $|W| > (k-1)\binom{2k-1}{k-1}$, there is a $k$-set $S$ in $W$ whose preimages are contained in at most $k-1$ caterpillar forests.

 This implies that $\psi$ restricted to $H[S \cup B_S]$ is an injective $\mc{Ca}$-cover of $K_{k,(k-1)^2 +1}$ of size at most $k-1$, which is impossible since $\ca(K_{k,(k-1)^2 +1}) = \left\lceil \frac{k(k-1)^2+k}{k + (k-1)^2}\right\rceil = k$, due to~\cite{Bei-65}.
\end{proof}

\subsection{Bounded (Simple) Tree-width}\label{subsec:tree-width}
A \emph{$k$-tree} is a graph that can be constructed starting with a $(k+1)$-clique and in every step attaching a new vertex to a $k$-clique of the already constructed graph. We use the term \emph{stacking} for this kind of attaching. The \emph{tree-width} $\tw(H)$ of a graph $H$ is the minimum $k$ such that $H$ is a \emph{partial $k$-tree}, i.e., $H$ is a subgraph of some $k$-tree~\cite{Rob-85}.

We consider a variation of tree-width, called \emph{simple tree-width}. A simple $k$-tree is a $k$-tree with the extra requirement that there is a construction sequence in which no two vertices are stacked onto the same $k$-clique. 
Now, the \emph{simple tree-width} $\stw(H)$ of $H$ is the minimum $k$ such that $H$ is a partial simple $k$-tree, i.e., $H$ is a subgraph of some simple $k$-tree.

For a graph $H$ with $\stw(H) = k$ or $\tw(H) = k$ we fix any (simple) $k$-tree that is a supergraph of $H$ and denote it by $\tilde{H}$. Clearly, $H$ inherits a construction sequence from $\tilde{H}$, where some edges are omitted.

\begin{lemma}\label{lem:twstw}
 We have $\tw(H)\leq \stw(H)\leq \tw(H)+1$ for every graph $H$.
\end{lemma}
\begin{proof}
 The first inequality is clear. For the second inequality we show that every $k$-tree $H$ is a subgraph of a simple $(k+1)$-tree $H$. Whenever in the construction sequence of $H$ several vertices $\{v_1,\ldots, v_n\}$ are stacked onto the same $k$-clique $C$ we consider $C\cup\{v_1\}$ as a $(k+1)$-clique in the construction sequence for $H$. Stacking $v_i$ onto $C$ now can be interpreted as stacking $v_i$ onto $C\cup\{v_{i-1}\}$ and omitting the edge $v_{i-1}v_i$. In this way we can avoid multiple stackings onto $k$-cliques by considering $(k+1)$-cliques.
\end{proof}

Simple tree-width endows the notion of tree-width with a more topological flavor. For a graph $H$ we have the following: $\stw(H)\leq 1$ if and only if $H$ is a linear forest, $\stw(H)\leq 2$ if and only if $H$ is outerplanar, $\stw(H)\leq 3$ if and only if $H$ is planar and $\tw(H)\leq 3$~\cite{Elm-90}. 

%
%
%

Simple tree-width also has connections to discrete geometry. In~\cite{Bel-04} a \emph{stacked polytope} was defined to be a polytope that admits a triangulation whose dual graph is a tree. From that paper one easily deduces that a full-dimensional polytope $P\subset\mathbb{R}^d$ is stacked if and only if $\stw(G_P)\leq d$. Here $G_P$ denotes the $1$-skeleton of $P$. See~\cite{Kna-12,Hel-13,Hel-14} for more on simple tree-width.

We consider both graphs with bounded tree-width and graphs with bounded simple tree-width as input classes, since {(A)} most of the results for outerplanar graphs are implied by the corresponding result for $\stw(H) \leq 2$, {(B)} lower bound results for $\stw(H) \leq 3$ carry over to planar graphs, {(C)} the extremal results for these two input classes differ when the covering class is that of interval graphs, and {(D)} when the maximum covering numbers are the same for both classes, the lower bounds are slightly stronger when witnessed by graphs of low {simple} tree-width.

\begin{theorem}\label{thm:tl-stw-upper}
We have $t_{\ell}(H)\leq \stw(H)$ for every graph $H$.
\end{theorem}
\begin{proof}
 If $\stw(H) =1$, then $H$ is a linear forest and hence an interval graph. If $\stw(H) =2$, then $H$ is outerplanar, and it even has track number at most $2$ as shown in~\cite{Kos-99}.
 
 So let $\stw(H)= s\geq 3$. We build an injective cover $\varphi: I_1\dotcup\cdots\dotcup I_k \to H$ with $|\varphi^{-1}(v)| \leq s$ for every $v \in V(H)$ and $I_i \in \mathcal{I}$ for $i \in [k]$. We use as $I_1,\ldots,I_k$ only certain interval graphs, which we call \emph{slugs}: A slug is like a caterpillar with a fixed spine, except that the graph $I_i^v$ induced by the leaves at every spine vertex $v \in I_i$ is a linear forest. (In a caterpillar $I_i^v$ is an independent set for every spine vertex $v$.) The end vertices of the spine are called \emph{spine-ends} and vertices of degree at most~$1$ in $I_i^v$ are called \emph{leaf-ends}. See the left of Figure~\ref{fig:ltrack-stw-upper} for an example of a slug $I_i$ with the spine drawn thick, spine-ends in white, and leaf-ends in gray. Note that slugs are indeed interval graphs.
 
 \begin{figure}[htb]
  \centering
  \includegraphics{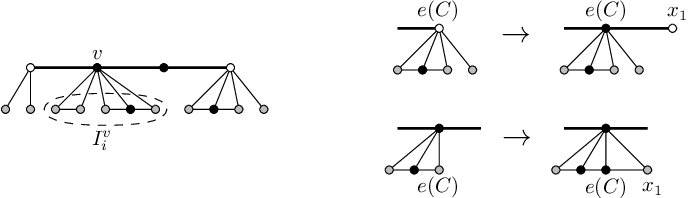}
  \caption{A slug and its extension.}
  \label{fig:ltrack-stw-upper}
 \end{figure}

 We define the cover $\varphi$ along a construction sequence of $H$ that is inherited from a simple $s$-tree $\tilde{H} \supseteq H$. At every step let $H'$ be the subgraph of $H$ that is already constructed and hence already covered by $\varphi$, and let $\tilde{H}'$ be the corresponding subgraph of $\tilde{H}$. We call an $s$-clique $C$ of $\tilde{H}'$ \emph{stackable} if no vertex has been stacked to $C$ so far. We maintain the following invariants on $\varphi$, which allow us to stack a new vertex onto every stackable $C$.
 
 \smallskip
  
 \noindent
 \textbf{Invariant.} At all times the following is satisfied for the current graph $H'$.
 \begin{enumerate}[label=\arabic*)]
  \item For every vertex $v$ in $H'$ there is a unique slug $I(v)$ with $I(v) \neq I(w)$ for $v\neq w$, and a spine vertex $s(v)$ of $I(v)$ in $\varphi^{-1}(v)$.\label{enum:slug}
  \item For every stackable $s$-clique $C$ there is a vertex $w_1 \in C$, a slug $I(C)$, and a spine-end or leaf-end $e(C)$ of $I(C)$ with $\varphi(e(C))= w_1$, such that:\label{enum:clique}
   \begin{enumerate}[label=2\alph*)]
    \item If $e(C)$ is a spine-end, then $I(C) \neq I(v)$ for all $v \in V(H')$.\label{enum:spine-end}
    \item If $e(C)$ is a leaf-end, then $I(C) = I(w_2)$ for some vertex $w_2 \in C\setminus\{w_1\}$, and the vertices $e(C)$ and $s(w_2)$ are adjacent in $I(C)$.\label{enum:leaf-end}
    \item Every leaf-end or spine-end $v$ is $e(C)$ for at most two cliques $C$ with equality only if $v$ has degree~$0$ or $1$ in the slug.\label{enum:at-most-2}
   \end{enumerate}
 \end{enumerate}
 
 It is not difficult to satisfy the above invariants for an initial $s$-clique of $\tilde{H}$. Indeed, this clique can be build up in a very similar way to the stacking procedure that we describe now: In the construction sequence of $H$ we are about to stack a vertex $w$ onto a stackable clique $C$ of the current graph $H'$. Let $ C= \{w_1,\ldots,w_s\}$. Without loss of generality we assume that $\varphi(e(C)) = w_1$ and that if $e(C)$ is a leaf-end, then $I(C) = I(w_2)$. We never change the preimages of vertices in $H$ under $\varphi$. In particular, all vertices we add to the existing or new slugs are mapped by $\varphi$ onto the new vertex $w$. We will denote these new vertices by $x_1,\ldots,x_s$ to emphasize that no more than $s$ such vertices are introduced. Note that for every $i \in [s]$, the clique $C_i$ in $\tilde{H}$ defined by $C_i = (C\setminus\{w_i\})\cup \{w\}$ is stackable in $H' \cup \{w\}$, and that all remaining stackable cliques in $H' \cup \{w\}$ are already stackable cliques in $H'$.
  
 For $i\in\{3,\ldots,s\}$ we do the following. If $ww_i \in E(H)$, then we introduce a new leaf $x_i$ to $I(w_i)$ at $s(w_i)$, and if $ww_i \notin E(H)$ we introduce a new slug consisting only of $x_i$. Either way, we set $e(C_{i-1}) = x_i$. Additionally we set $e(C_1) = x_s$. Note that~\ref{enum:leaf-end} is satisfied since $w_i,w \in C_{i-1}$ and $w_s,w \in C_1$.
 
 It remains to cover possible edges joining $w$ to $\{w_1,w_2\}$, to find a spine-end or leaf-end $e(C_s)$ for $C_s$, and to find a slug $I(w)$ for the new vertex $w$. In doing so we may still introduce two new vertices $x_1$ and $x_2$ to our slugs. We distinguish two cases, which are illustrated on the right in Figure~\ref{fig:ltrack-stw-upper}.
 
 \begin{description}
  \itemsep\smallskipamount
  \item[Case 1:] If $e(C)$ is a spine-end of $I(C)$, then we first proceed with $w_2$ similarly as with $w_i$ for $i\geq 3$ above. That is, we introduce a new leaf $x_2$ at $s(w_2)$ if $ww_2 \in E(H)$ and a new slug consisting only of $x_2$ if $ww_2 \notin E(H)$, and we set $e(C_s) = x_2$.
  
  \smallskip
  
  \begin{description}
   \item[Case 1.1:] If $ww_1 \in E(H)$, then we introduce a new spine vertex $x_1$ to $I(C)$ adjacent to $e(C)$. This covers the edge $ww_1$, since we assumed that $\varphi(e(C)) = w_1$. We set $I(w) = I(C)$, which satisfies condition~\ref{enum:slug} of the invariant since~\ref{enum:spine-end} implies $I(C) \neq I(v)$ for every vertex $v$ in $H'$.
   
   \item[Case 1.2:] If $ww_1 \notin E(H)$, then  we introduce a new slug $I$ consisting only of $x_1$ and set $I(w) = I$. 
  \end{description}

  \item[Case 2:] If $e(C)$ is a leaf-end of $I(C)$, then by assumption we have $I(C) = I(w_2)$.
  
  \smallskip
  
  \begin{description}
   \item[Case 2.1:] If $ww_2 \in E(H)$, then  we introduce a new leaf $x_2$ to $I(C)$ adjacent to $s(w_2)$ and a new slug $I$ consisting just of a new vertex $x_1$. If additionally $ww_1 \in E(H)$, then we also introduce an edge joining $x_2$ and $e(C)$ in $I(C)$. Again, since $\varphi(e(C)) = w_1$ and $\varphi(x_2) = w$, this covers the edge $ww_1$. Either way, we set $e(C_s) = x_2$ and $I(w) = I$.
   
   \item[Case 2.2:] If $ww_2 \notin E(H)$, then we introduce a new slug $I$ consisting only of a new vertex $x_2$ and set $I(w) = I$. When $ww_1 \in E(H)$ we add a new leaf $x_1$ to $s(w_1)$ in $I(w_1)$, and when $ww_1 \notin E(H)$, then we introduce a new slug consisting only of $x_1$. Either way we set $e(C_s) = x_1$.
  \end{description}
 \end{description}
 
 It is straightforward to check that we obtain a $\mathcal{I}$-cover of $H'\cup \{w\}$ and that the invariants above are satisfied. Note that since $\tilde{H}$ is a {simple} $s$-tree, the clique $C$ is no longer stackable and hence condition~\ref{enum:clique} of the invariant need not be satisfied in $H'\cup \{w\}$. Finally, every stackable clique in $H'$ different from $C$ was not affected by the above procedure, which completes the proof.
\end{proof}

We can prove three lower bounds for covering numbers. 

\begin{theorem}\label{thm:i-tw-lower}
For $k\geq 1$, there is a bipartite graph $H$ such that
\vspace{-0.8em}
\[
 \stw(H)\leq \tw(H)+1\leq k+1\leq \fca(H) = i(H).
\]
\end{theorem}
\begin{proof}
 Construct $H$ from $K_{k,n}$ with $n=2k^2+1$ by adding a pendant vertex at each vertex of the larger partite set $B$. It is easy to see that $\tw(H)\leq k$, and then Lemma~\ref{lem:twstw} yields $\stw(H)\leq \tw(H)+1$. 

 Consider any $\mc{Ca}$-cover $\varphi$ of $H$ with $s = \max \{|\varphi^{-1}(v)| : v\in V(H)\}$ and its restriction $\psi$ to the subgraph $H'$ of $H$ obtained by removing all edges of $K_{k,n}$. By Lemma~\ref{lem:Kmn-caterpillar} there are at least $n - 2sk = 2k(k-s) + 1$ vertices $b \in B$ such that $|\psi^{-1}(b)| \leq |\varphi^{-1}(b)| - k$. Any such $b$ is incident to an edge in $H\setminus K_{k,n}$, which should be covered by $\psi$. Thus, $|\psi^{-1}(b)| \geq 1$. Hence, $s \geq |\varphi^{-1}(b)| \geq k+1$, so $\fca(H) \geq k+1$.
\end{proof}

\begin{theorem}\label{thm:t-stw-lower}
For $k\geq 3$, there is a bipartite graph $H$ such that
\vspace{-0.8em}
\[
 {\stw}(H)+1\leq k+1 \leq \ca(H) = t(H).
\]
\end{theorem}
\begin{proof}
 The construction of the graph $H$ starts with $H_0\cong K_{k-1,m_1}$, where $|B| = m_1 = 2(2k^2-2k+1)$. Let $B = \{u_1,...,u_{m_1/2}\}\cup \{v_1,...,v_{m_1/2}\}$.  For $i\in [m_1/2]$, add a copy $I_i$ of $K_{2,5k-5}$ with partite sets $\{u_i,v_i\}$ and $\{b_1^{i,j},...,b_{k-1}^{i,j}\colon\,j\in[5]\}$, calling the smaller set $A_i$ and the larger set $B_i$.  Next, let $m_2=(k-2)^+1$.  For $(i,j)\in[m_1/2]\times[5]$, add a set $B_{i,j}$ of $m_2$ new vertices and a copy $J_{i,j}$ of $K_{k-1,m_2}$ with partite sets ${b_1^{i,j},...,b_{k-1}^{i,j}}$ and $B_{i,j}$.  Note that the smaller part $A_{i,j}$ in $J_{i,j}$ is contained in $B_i$.
See Figure~\ref{fig:track-stw-low} for an illustration.

 \begin{figure}[htb]
  \centering
  \psfrag{G=}[rr]{$H:$}
  \psfrag{H1}{$I_1$}
  \psfrag{Hm-1}{$I_{\frac{m_1-2}{2}}$}
  \psfrag{Hm}{$I_{\frac{m_1}{2}}$}
  \psfrag{H=}{$I_i=$}
  \psfrag{k}[rr]{$k-1$}
  \psfrag{k-1}[cc]{$k-1$}
  \psfrag{m1}{$m_1$}
  \psfrag{m2}{$5(k-1)$}
  \psfrag{u}{$u_i$}
  \psfrag{v}{$v_i$}
  \psfrag{J1}{$J_{i1}$}
  \psfrag{J4}{$J_{i4}$}
  \psfrag{J5}{$J_{i5}$}
  \includegraphics[width=.9\textwidth]{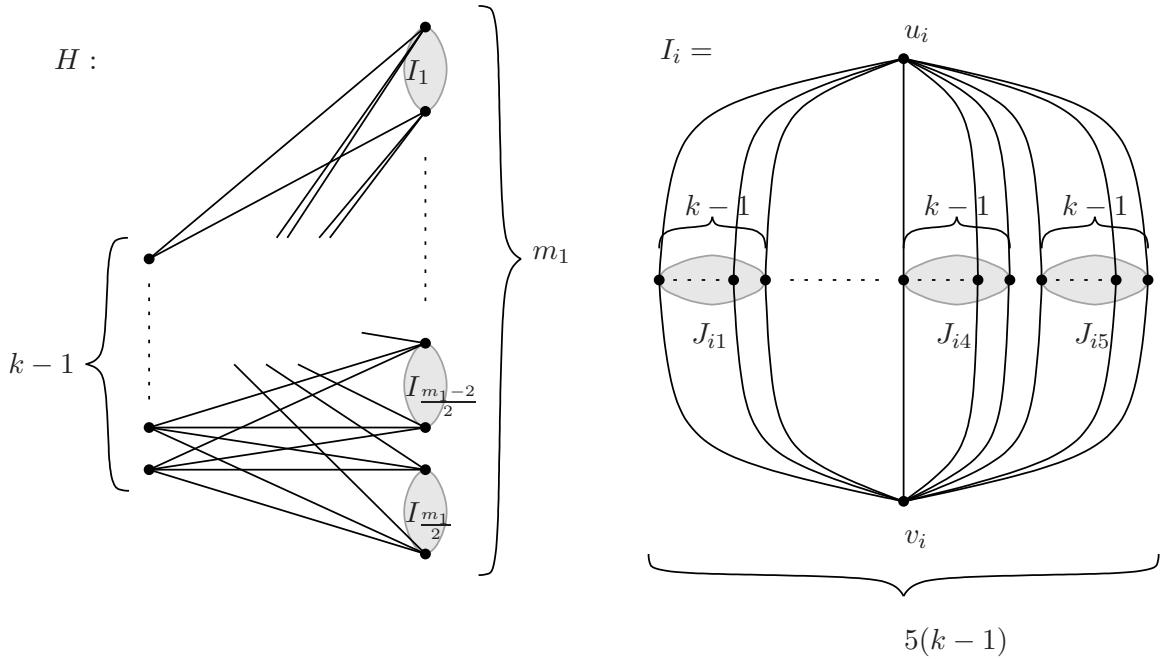}
  \caption{The graph $H$ and its induced subgraph $I_i$ and $J_j$.}
  \label{fig:track-stw-low}
 \end{figure}

 Assume for the sake of contradiction that $\varphi$ is an injective $\mc{Ca}$-cover of $H$ of size at most $k$. Consider the restriction $\psi$ of $\varphi$ to the subgraph $H' = H \setminus E(H_0)$ of $H$. By Lemma~\ref{lem:Kmn-caterpillar} there are at least $m_1 - 2k(k-1) = 2k^2-2k+2 > \frac{m_1}{2}$ vertices in $b \in B$ with $|\psi^{-1}(b)| \leq 1$. In particular there is some $i' \in [\frac{m_1}{2}]$ such that $|\psi^{-1}(u_i')|,|\psi^{-1}(v_i')| \leq 1$. That is, in the covering $u_i'$ and $v_i'$ each appear in only one caterpillar forest, which we call $C_{u_i'}$ containing $u_i$ and $C_{v_i'}$ containing $v_i'$. 
%
 Now consider the restriction $\phi$ of $\psi$ to the subgraph $H'' = H' \setminus E(I_{i'})$ of $H'$. Again by Lemma~\ref{lem:Kmn-caterpillar} there are at least $5(k-1)-4$ vertices $b\in B_{i'}$ with $|\phi^{-1}(b)| \leq k-2$. In particular there is some $j' \in [5]$ such that $|\phi^{-1}(b)| \leq k-2$ for all $b \in A_{i'j'}$.
 
 In other words, $\phi$ restricted to $H[A_{i'j'} \cup B_{i'j'}]$ is an injective $\mc{Ca}$-cover of $K_{k-1,(k-2)^2 +1}$ of size at most $k-2$, which is impossible, since $\ca(K_{k-1,(k-2)^2 +1}) = \left\lceil \frac{(k-1)(k-2)^2+k-1}{k-1 + (k-2)^2}\right\rceil = k-1$, due to~\cite{Bei-65}.

 It remains to show that $\stw(H)\leq k$. In order to describe the construction sequence for a simple $k$-tree containing $H$, we introduce some further vertex labels. Let $A_0 = \{a_1,\ldots,a_{k-1}\}$ be the smaller partite set of $H_0$, recall that $B_i = A_{i1}\cup\ldots\cup A_{i5}$ where $A_{ij}=\{b^{ij}_1,\ldots,b^{ij}_{k-1}\}$ for all $i\in[\frac{m_1}{2}],j\in[5]$, and let $B_{ij} = \{c^{ij}_1,\ldots,c^{ij}_{m_2}\}$ for $i \in [\frac{m_1}{2}],j\in[5]$. We construct a simple $k$-tree starting with a $(k+1)$-clique on $A \cup \{u_1,v_1\}$ via the following stackings:
 \begin{itemize}
  \item[A] Stack $u_i$ onto $A \cup \{v_{i-1}\}$ and $v_i$ onto $A \cup \{u_i\}$ \hfill $\forall i\in\{2,\ldots,\frac{m_1}{2}\}$
  \item[B] Stack $b^{i1}_\ell$ onto $\{a_1,\ldots,a_{k-\ell-1},u_i,v_i,b^{i1}_1,\dots,b^{i1}_{\ell-1} \}$ \hfill $\forall i \in [\frac{m_1}{2}],\ell \in [k-1]$
  \item[C] Stack $b^{ij}_\ell$ onto $\{u_i,v_i,b^{i(j-1)}_1,\dots,b^{i(j-1)}_{k-\ell-1},b^{ij}_1,\ldots,b^{ij}_{\ell-1} \}$ \\ \hspace*{0pt} \hfill $\forall i \in [\frac{m_1}{2}],\ell \in [k-1],j \geq 2$
  \item[D] Stack $c^{ij}_1$ onto $A_{ij} \cup \{u_i\}$ and $c^{ij}_\ell$ onto $A_{ij} \cup \{c^{ij}_{\ell-1}\}$ \hfill $\forall i \in [\frac{m_1}{2}],j\in[5],\ell \geq 2$
 \end{itemize}
 One can check that after step A. the entire graph $H_0$ is contained in the so-far constructed $k$-tree. Step B. deals with the complete bipartite graphs induced on $\{u_i,v_i\}\cup A_{i1}$ for all $i\in[\frac{m_1}{2}]$, step C. adds the remaining complete bipartite graphs induced on $\{u_i,v_i\}\cup A_{ij}$ for $j\geq 2$, such that afterwards all $I_{i}$ are contained. In step D. all edges and vertices necessary for the $J_{ij}$ are created. 
 Since no $k$-clique appears twice we conclude that $\stw(H) \leq k$.
\end{proof}

\begin{theorem}\label{thm:fca-stw-lower}
 For $k \geq 2$, there is a graph $H$ such that
 \vspace{-0.8em}
 \[
  \stw(H)+1 \leq k+1 \leq \fca(H).
 \]
\end{theorem}
\begin{proof}
 Fix $k\geq 2$. We construct $H$ starting with a star with $k-1$ leaves $\ell_1,\ldots,\ell_{k-1}$ and center $c_1$. In the simple partial $k$-tree containing $H$ this star is a $k$-clique. For $n = 16k^2-16k+4$ and $2\leq i \leq n$ stack a new vertex $c_i$ to $\ell_1,\ldots,\ell_{k-1},c_{i-1}$. Now stack vertices $s_2,\ldots, s_n$ to $\ell_1,\ldots,\ell_{k-2},c_{i-1}, c_i$. Finally introduce a pendant vertex $a_i$ as a neighbor of $s_i$, for each $i$. In the simple partial $k$-tree containing $H$, the vertex $a_i$ is stacked to the $k$-clique on $\ell_1,\ldots,\ell_{k-2},c_{i-1},s_i$. By construction $\stw(H)\leq k$. See Figure~\ref{fig:caf-stw-upper} for an illustration.

 \begin{figure}[htb]
  \centering
  \includegraphics[width=.7\textwidth]{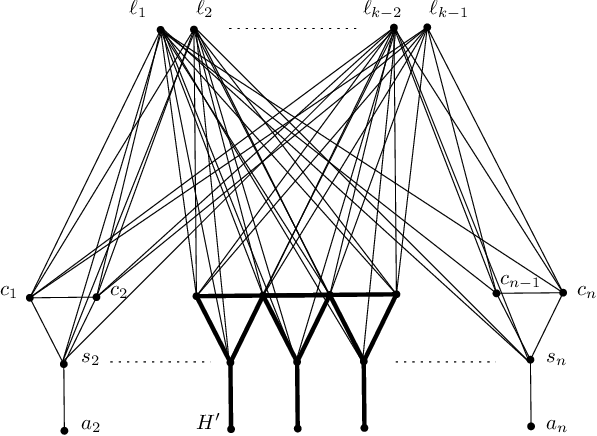}
  \caption{The graph $H$ and its subgraph $H'$.}
  \label{fig:caf-stw-upper}
 \end{figure}

 Assume for the sake of contradiction that $\fca(H)\leq k$. That is, there is a $\mc{Ca}$-cover $\varphi$ of $H$ with $|\varphi^{-1}(v)| \leq k$ for all $v\in V(H)$. We consider three edge-disjoint complete bipartite subgraphs $H_1, H_2, H_3$ of $H$ with partite sets $A_i$ and $B_i$ for $H_i$ defined as follows:
 \begin{itemize}
  \item $A_1=\ell_1,\ldots,\ell_{k-1}$ and $B_1=\{c_{2i}\colon\ 1\leq i\leq n/2\}$
  \item $A_2=\ell_1,\ldots,\ell_{k-1}$ and $B_2=\{c_{2i-1}\colon\ 1\leq i\leq n/2\}$
  \item $A_3=\ell_1,\ldots,\ell_{k-2}$ and $B_3=\{s_i\colon\ 2\leq i\leq n\}$
 \end{itemize}
 
 Note that $H_i$ and $H_j$ are edge-disjoint for $i \neq j$. Denote by $\psi$ the restriction of $\varphi$ to $H \setminus (E(H_1) \cup E(H_2) \cup E(H_3))$. We apply Lemma~\ref{lem:Kmn-caterpillar} three times, once for each $H_i$, but the bounds for restrictions of $\varphi$ to $H \setminus E(H_i)$ clearly also apply to $\psi$. Thus, we obtain sets $W_i\subset B_i$ (for each $i\in\{1,2,3\}$). For $i\in\{1,2\}$ we get $|W_i| \geq n/2-2k(k-1)$ and $\psi^{-1}(b)\leq k-(k-1)=1$ for $b\in W_i$. Furthermore we have $|W_3| \geq n-1-2k(k-2)$ and $\psi^{-1}(b)\leq k-(k-2)=2$ for $b\in W_3$. From the choice of $n$ it follows that there exist  $c_i,c_{i+1},c_{i+2},c_{i+3}\in W_1\cup W_2$ with consecutive indices such that $s_{i+1},s_{i+2},s_{i+3}\in W_3$. Together with the leaves $a_{i+1},a_{i+2},a_{i+3}$ these vertices induce a $10$-vertex graph $H'$ highlighted in Figure~\ref{fig:caf-stw-upper}. It is not difficult to check that there is no $\mc{Ca}$-cover $\psi$ of $H'$ with $|\psi^{-1}(c_{i+j})|\leq 1$ for $j\in\{0,1,2,3\}$ and $|\psi^{-1}(s_{i+j}
)|\leq 2$ for $j\in\{1,2,3\}$ --- a contradiction.
\end{proof}

\subsection{Planar and Outerplanar Graphs}\label{subsec:planar}
Determining maximum covering numbers of (bipartite) planar graphs and outerplanar graphs enjoys a certain popularity, as demonstrated by the variety of citations in Table~\ref{table:1}. We add three easy new results to the list.

%

\begin{corollary}\label{cor:loc-star-planar}
 The star arboricity of bipartite planar graphs is at most $4$.
 The local star arboricity of planar graphs and bipartite planar graphs is at most $4$ and at most $3$, respectively.
\end{corollary}
\begin{proof}
 As mentioned in Section~\ref{subsec:star}, the arboricity $a(H)$ of every graph $H$ can be expressed as  $\max_{S\subseteq V(H)}\left\lceil\frac{|E[S]|}{|S|-1}\right\rceil$~\cite{Nas-64}. By Euler's Formula every planar graph has at most $3V(H)-6$ edges and every bipartite planar graph has at most $2V(H)-4$ edges and clearly both classes are closed under taking subgraphs. Together it follows that every planar graph has arboricity at most~$3$ and every planar bipartite graph has arboricity at most~$2$. With this, the statement about global star arboricity follows since we have $\sa(H)\leq 2a(H)$ by~\cite{Alo-92}. The statements about local arboricity follow since we have $\lsa(H)\leq a(H)+1$ by Theorem~\ref{thm:starorient}.
\end{proof}

The only question mark in Table~\ref{table:1} concerns the local track number of planar graphs. Scheinerman and West~\cite{Sch-83} show that the interval number of planar graphs is at most~$3$, but this is verified with a cover that is not injective. On the other hand, there are bipartite planar graphs with track number~$4$~\cite{Gon-09}. However by Corollary~\ref{cor:loc-star-planar} and Theorem~\ref{thm:tl-stw-upper} every bipartite planar graph and every planar graph of tree-width at most~$3$ has local track number at most~$3$. We believe that there are planar graphs with local track number $4$, but the following remains open:

\begin{quest}\label{quest:localtrackplanar}
 What is the maximum local track number of a planar graphs?
\end{quest}

\section{Separability and Complexity}\label{sec:complexity}

This section is devoted to different types of questions. First, we investigate how much global, local, and folded covering numbers can differ with respect to the same covering and input class. Second, we look at the complexity of computing these parameters.

In Table~\ref{table:1} we provide several pairs of an input class $\mc{H}$ and a covering class $\mc{G}$ for which the global covering number and the local covering number differ, i.e., $c_{g}^{\mc{G}}(\mc{H}) > c_{\ell}^{\mc{G}}(\mc{H})$. Indeed this difference can be arbitrarily large.

\begin{theorem}\label{thm:separate-gl}
 For the covering class $\mc{Q}$ of collections of cliques and the input class $\mc{H}$ of line graphs, we have $c_{g}^{\mc{Q}}(\mc{H}) = \infty$ and $c_{\ell}^{\mc{Q}}(\mc{H}) \leq 2$.
\end{theorem}
\begin{proof}
 By a result of Whitney~\cite{Whi-32} a graph $H$ is a line graph if and only if $c_{\ell}^{\mc{Q}}(H) \leq 2$.
 
 To prove $c_{g}^{\mc{Q}}(\mc{H}) = \infty$, we claim that $c_{g}^{\mc{Q}}(L(K_n)) \in \Omega(\log n)$, i.e., the covering number of the line graph of the complete graph on $n$ vertices is unbounded as $n$ goes to infinity. Assume that $L(K_n)$ is covered by $k$ collections of cliques $C_1,\ldots,C_k$. Every clique in $L(K_n)$ corresponds to either a triangle or a star in $K_n$. Now, every $C_i$ in $L(K_n)$ corresponds to a vertex disjoint collection of triangles and stars in $K_n$. Together these collections cover the edges of $K_n$. We will restrict the covering of $L(K_n)$ to a covering of $L(K_m)$ with collections of cliques all of whose cliques correspond to stars in $K_m$. In the first step delete at most $\frac{1}{3}n$ vertices of $K_n$ such that in the restricted cover of the smaller line graph no clique in $C_1$ corresponds to a triangle. Repeating this for every $C_i$, we end up with a clique cover of $L(K_m)$ with $m \geq (\frac{2}{3})^k n$ that corresponds to a cover of $K_m$ with star forests. 
Since by~\cite{Aki-85} the star arboricity of $K_m$ is $\left\lceil\frac{m}{2}\right\rceil+1$, we get $k \geq \frac{m+2}{2} > (\frac{2}{3})^{k-1} n$, and thus $k \in \Omega(\log n)$.
\end{proof}

\begin{remark}
 Milans, Stolee, and West~\cite{Mil-12} proved a similar result with interval graphs as covering class, i.e., they showed that the growth rate of $t(L(K_n))$ is between $\Omega(\log\log n/\log\log\log n)$ and $O(\log\log n)$, while $i(H) \leq 2$ for every line graph $H$.
\end{remark}

A case of particular interest to us is the input class of {claw-free graphs} -- a class containing line graphs. It has been shown that this class has unbounded local clique covering number~\cite{Jav-12}. We conjecture the following stronger statement:

\begin{conjecture}
 The class of claw-free graphs has unbounded interval number.
\end{conjecture}

What can be said about local and folded covering number? Table~\ref{table:1} suggests that the separation of the local and the folded covering number is more difficult. Indeed we have $c_{\ell}^{\mc{G}}(\mc{H}) = c_{f}^{\mc{G}}(\mc{H})$ for every $\mc{G}$ and $\mc{H}$ in Table~\ref{table:1}, except for the local track number of planar graphs, (c.f. Question~\ref{quest:localtrackplanar}). However, proving upper bounds for $c_{\ell}^{\mc{G}}(\mc{H})$ can be significantly more elaborate than for $c_{f}^{\mc{G}}(\mc{H})$, even if we suspect that both values are equal; see for example Conjecture~\ref{conj:LLAC} and Theorem~\ref{thm:FLAC}.

Observing that there is {no} injective cover of a path by cycles of length at least $3$ and that every path is the homomorphic image of a cycle one gets:

\begin{obs}\label{obs:separate-lf}
 For the covering class $\mc{C}$ of collections of cycles of length at least $3$ and the input class $\mc{H}$ of paths, we have $c_{\ell}^{\mc{C}}(\mc{H}) = \infty$ and $c_{f}^{\mc{C}}(\mc{H}) \leq 2$.
\end{obs}

Observation~\ref{obs:separate-lf} may be considered pathological. However, the local and folded covering number may differ also when $c_{\ell}^{\mc{G}}(H) < \infty$. We gave one example for this when considering coverings of the Petersen graph with disjoint unions of cycles, see Proposition~\ref{prop:petersen}. Here is another example: It is known that $i(K_{m,n}) = \left\lceil \frac{mn+1}{m+n} \right\rceil$~\cite{Har-79} and $t(K_{m,n}) = \left\lceil \frac{mn}{m+n-1} \right\rceil$~\cite{Gya-95}. The lower bound on $t(K_{m,n})$ presented in~\cite{Deo-94} indeed gives $t_{\ell}(K_{m,n}) \geq \left\lceil \frac{mn}{m+n-1} \right\rceil$ and hence we have $t_{\ell}(K_{m,n}) > i(K_{m,n})$ for appropriate numbers $m$ and $n$, such as $n = m^2-2m+2$. With Proposition~\ref{prop:basic} this translates into $\lca(K_{m,n}) > \fca(K_{m,n})$. Apart from these examples, we have no general answer to the following question.

\begin{quest}
 By how much can folded and local covering number differ?
\end{quest}

\medskip

Another interesting aspect of covering numbers concerns the computational complexity of determining them. Very informally, one might suspect that the computation of $c_{f}^{\mc{G}}(H)$ is easier than of $c_{\ell}^{\mc{G}}(H)$, which in turn is easier than computing $c_{g}^{\mc{G}}(H)$. For example, if $\mc{M}$ is the class of all matchings, then $c_{g}^{\mc{M}}(H) = \chi'(H)$, the edge-chromatic number of $H$. Hence deciding $c_{g}^{\mc{M}}(H) \leq 3$ is NP-complete even for $3$-regular graphs~\cite{Hol-81}. On the other hand $c_{\ell}^{\mc{M}}(H)$ equals the maximum degree of $H$ and can therefore be determined very efficiently. As a second example, more elaborate, consider the star arboricity $\sa(H)$ and the caterpillar arboricity $\ca(H)$. Deciding $\sa(H) \leq k$~\cite{Hak-96,Gon-09} and deciding $\ca(H) \leq k$~\cite{Gon-09,She-96} are NP-complete for $k=2,3$. The complexity for $k\geq 4$ is unknown in both cases. To the best of our knowledge, the complexity of determining the local and folded caterpillar 
arboricity of a graph is also open.
 On the other hand, from Theorem~\ref{thm:starorient} we can derive the following.

\begin{theorem}\label{thm:lsa-in-P}
 The local star arboricity can be computed in polynomial-time.
\end{theorem}
\begin{proof}
 In~\cite{Fra-78} a flow algorithm is used that given a graph $H$ and $\alpha:V(H)\to\mathbb{N}$ decides if an orientation $D$ of $H$ exists such that the out-degree of $v$ in $D$ is at most $\alpha(v)$ for all $v\in V(H)$. Moreover, if such a $D$ exists the algorithm finds one minimizing the maximum out-degree. Now by Lemma~\ref{lem:pseudo}, we may use this algorithm to find $p(H)$ in polynomial-time. Now let $\alpha(v) = p(H)$ whenever $v$ has degree $p(H)$ and $\alpha(v) = p(H)-1$ otherwise. We use the algorithm of~\cite{Fra-78} to check if an orientation $D$ of $H$ satisfying the out-degree constraints given by $\alpha$ exists. By Theorem~\ref{thm:starorient} we have $\lsa(H)=p(H)$ if and only if there exists such an orientation and $\lsa(H)=p(H)+1$ otherwise.
\end{proof}

Finally, consider interval graphs as the covering class. Shmoys and West~\cite{Shm-84} and Jiang~\cite{Jia-10} showed that deciding $i(H) \leq k$ and deciding $t(H)\leq k$ are NP-complete for every $k \geq 2$, respectively. We claim that the reduction of Jiang also holds for the local track number. 

\begin{quest}
 Are there a covering and an input class for which the computation of the folded or local covering number is NP-complete while the global covering number can be computed in polynomial-time?
\end{quest}

\section{Concluding remarks}\label{sec:fur}
We have presented new ways to cover a graph and given many example covering classes. Also, we highlighted some conjectures and questions on the way, such as the question whether the maximum track number of planar graphs is $3$ or $4$ (Question~\ref{quest:localtrackplanar}). 

One conjecture important to us is LLAC (Conjecture~\ref{conj:LLAC}), which is a weakening of the linear arboricity conjecture (LAC). Besides LLAC, there are several more weakenings of LAC that are still open. For example it is open, whether the caterpillar arboricity of graph $H$ of maximum degree $\Delta(H)$ is always at most $\left\lceil\frac{\Delta(H)+1}{2}\right\rceil$. Yet a weaker, but still open, question asks whether the track number of $H$ is always at most $\left\lceil\frac{\Delta(H)+1}{2}\right\rceil$. As a positive result, by Theorem~\ref{thm:starorient} one obtains that for a regular graph $H$ of even degree the local star-arboricity is $\left\lceil\frac{\Delta(H)+1}{2}\right\rceil$, which in particular settles the question for local caterpillar arboricity and local track number for such input graphs. On the other hand, Theorem~\ref{thm:starorient} also tells us that in a regular graph $H$ of odd degree the local star arboricity is larger than $\left\lceil\frac{\Delta(H)+1}{2}\right\rceil$. To the best of our knowledge, it is open whether the local caterpillar arboricity or local track number of such a graph $H$ is always at most $\left\lceil\frac{\Delta(H)+1}{2}\right\rceil$.

Apart from the problems already mentioned throughout the paper, it is interesting to consider the local and folded variants for more graph covering problems from the literature. For example the covering number with respect to planar and outerplanar graphs is known as the \emph{thickness} and \emph{outerthickness}~\cite{Bei-69}, respectively, and the folded covering number with respect to planar graphs is called the \emph{splitting number}~\cite{Jac-85}. The local covering number in these cases seems unexplored. Further interesting covering classes include linear forests of bounded length~\cite{Alo-01}, forests of stars and triangles~\cite{Fia-07}, and chordal graphs.

A concept dual to covering is \emph{packing}. For an input graph $H$ and a class $\mc{G}$ of \emph{packing graphs}, we define a \emph{$\mc{G}$-packing of $H$} to be an edge-injective homomorphism $\varphi$ to $H$ from the disjoint union $G_1 \dotcup G_2 \dotcup \cdots \dotcup G_k$ with $G_i \in \mc{G}$ for $i \in [k]$. The \emph{size} of a packing is the number of packing graphs in the disjoint union. A packing $\varphi$ is \emph{injective} if $\varphi\rst{G_i}$, that is, $\varphi$ restricted to $G_i$, is injective for every $i \in [k]$.

\begin{definition}\label{defn:dualnumbers}
 For a packing class $\mc{G}$ and an input graph $H = (V,E)$ define the \emph{(global) packing number} $p_{g}^{\mc{G}}(H)$, the \emph{local packing number} $p_{\ell}^{\mc{G}}(H)$, and the \emph{folded packing number} $p_{f}^{\mc{G}}(H)$ as follows:
 \begin{itemize}[label= ]
  \item $p_{g}^{\mc{G}}(H) = \max\left\{ \text{size of }\varphi : \varphi \text{ is an injective }\mc{G}\text{-packing of }H\right\}$
  \item $p_{\ell}^{\mc{G}}(H) = \max\left\{ \min_{v \in V} |\varphi^{-1}(v)| : \varphi \text{ is an injective }\mc{G}\text{-packing of }H\right\}$
  \item $p_{f}^{\mc{G}}(H) = \max\left\{ \min_{v \in V} |\varphi^{-1}(v)| : \varphi \text{ is a }\mc{G}\text{-packing of } H \text{ having size } 1\right\}$
 \end{itemize}
\end{definition}

Let us rephrase $p_{g}^{\mc{G}}(H)$, $p_{\ell}^{\mc{G}}(H)$, and $p_{f}^{\mc{G}}(H)$: The packing number is the maximum number of packing graphs that can be packed into the input graph, where packing means identifying edge-disjoint subgraphs in $H$ that lie in $\mc{G}$. The local packing number does not measure the number of packing graphs in a packing; instead the minimum number of graphs packed at any one vertex is maximized. The folded packing number is the maximum $k$ such that every vertex $v$ of $H$ can be split into $k$ vertices, distributing the incident edges at $v$ arbitrarily (not repeatedly) among them, such that the resulting graph is in $\mc{G}$. Two classical packing problems are given by $\mc{G}$ being the class of non-planar graphs or non-outerplanar graphs. In this case the global packing numbers are called \emph{coarseness} and \emph{outercoarseness}~\cite{Bei-69}, respectively.

\bigskip

{\bf Acknowledgments:}
We thank Marie Albenque, Daniel Heldt, and Bartosz Walczak for fruitful discussions and two anonymous referees and Douglas B. West for useful comments improving the presentation of the paper.
Kolja Knauer was partially supported by DFG grant FE-340/8-1 as part of ESF project GraDR EUROGIGA and PEPS grant EROS and Torsten Ueckerdt by GraDR EUROGIGA project No. GIG/11/E023.
\bibliography{lit}
\bibliographystyle{amsplain}

\end{document}